\documentclass[11pt,reqno]{amsart} 

\usepackage{amssymb,latexsym}
\usepackage{cite} 

\usepackage[height=190mm,width=130mm]{geometry} 

\theoremstyle{plain}
\newtheorem{theorem}{Theorem}

\newtheorem{corollary}{Corollary}

\theoremstyle{definition}

\theoremstyle{remark}

\numberwithin{equation}{section} 

\begin{document}
\title[Betti and Tachibana numbers]{Betti and Tachibana numbers} 

\author{S.E. Stepanov}
\address{Finance University under the Government of Russian Federation\\ 
 Dept.~of Mathe\-ma\-tics\\ Leningradsky Prospect, 49-55\\ 125993 \\ Moscow\\ Russia}

\email{s.s.stepanova@mail.ru}

\thanks{The paper was supported by grant P201/11/0356 of The Czech Science Foundation.}

\author{J. Mike\v s}

\address{Palacky University\\Faculty of Science\\ Dept.~of Algebra and Geometry\\
17.~listopadu 12\\ 77146 \\ Olomouc\\ Czech Republic}


\email{josef.mikes@upol.cz}

\begin{abstract}
We present a rough classification of differential forms on a Riemannian manifold, we consider definitions and properties of conformal Killing forms on a compact Riemannian manifold and define Tachibana numbers as an analog of the well known Betti numbers. We state the conditions that characterize these numbers. In the last section we show connections between the Betti and Tachibana numbers.
\end{abstract}




\maketitle

\def\a{\alpha}
\def\b{\beta}
\def\la{\lambda}
\def\om{\omega}
\def\op{\omega^\prime}
\def\ds{\displaystyle}
\def\noi{\noindent}
\def\R{{\mathbb R}}
\def\RR{{\mathfrak R}}
\def\tm{\hbox{${\bold T}^r(M,{\mathbb R})$}}
\def\dm{\hbox{${\bold D}^r(M,{\mathbb R})$}}
\def\fm{\hbox{${\bold F}^r(M,{\mathbb R})$}}
\def\hm{\hbox{${\bold H}^r(M,{\mathbb R})$}}
\def\km{\hbox{${\bold K}^r(M,{\mathbb R})$}}
\def\cm{\hbox{${\bold C}^r(M,{\mathbb R})$}}
\def\omm{\hbox{${\Omega}^r(M)$}}
\def\pmm{\hbox{${\bold P}^r(M,{\mathbb R})$}}

\section{Introduction}
The the paper is based on our plenary lecture at the International Conference AGMP-8 (Brno, 12--14 September 2012). The  paper is organized as follows. In the first section, we introduce some notations and give some basic definitions of the theory of differential forms on a Riemannian manifold $(M, g)$ and natural operators on forms. In addition, we present a rough classification of differential forms on a Riemannian manifold. In the next two paragraphs of the paper we consider properties of harmonic and conformal Killing forms on an $n$-dimensional compact Riemannian manifold $(M, g)$ and determine the Tachibana number $t_r(M)$ as an analog of the Betti number $b_r(M)$ of $(M, g)$ for 
$1\leq r \leq n - 1$. Moreover, we state some the conditions that characterize these numbers. In the last section we formulate results on relationships between Betti and Tachibana numbers.

\section{On a classification of differential forms on Riemannian manifold}

In this section we will show a rough classification of differential forms on Riemannian manifold which will be useful for the establishment of desired results. Moreover, we want to fix some notations and basic concepts. 

Let $ (M, g)$ be an $n$-dimensional compact and orientable Riemannian manifold with the Live-Civita connection $\nabla$  and let $\Omega^r(M)$  denote the vector space of smooth $r$-forms on $M$  for  $1\leq r\leq n-1$. We define an inner product on $\Omega^r(M)$  by the formula:  
$\langle\om,\op\rangle=\int_M g(\om,\op)\,d\,{\rm vol}$
 for any $\om,\op\in\Omega^r(M)$. The inner product structures on $\Omega^r(M)$  allow us to define the exterior co-differential operator $d^*$: $\Omega^r(M)\to\Omega^{r- 1}(M)$ as a formal adjoint to the well known exterior differential operator $d$: $\Omega^r(M)\to\Omega^{r+1}(M)$ by the following formula   $\langle d\om,\theta\rangle=\langle\om,d^*\theta\rangle$.
                          
More than thirty years ago, Bourguignon (see\cite{1}) considered the space of Riemannian natural (with respect to isometric diffeomorphisms) first-order differential operators on $\Omega^r(M)$  with values in the space of homogeneous tensors on $M$.  He proved the existence of a basis of this space which consists of three operators $\{D_1,D_2,D_3\}$  where $D_1:= d$ and $D_2:= d^*$. As for the third operator $D_3$, Bourguignon said that $D_3$ does not have any geometric interpretation for $r > 1$. It was also pointed out that in the case $r = 1$ the kernel of $D_3$ consists of infinitesimal conformal transformations of $M$.

         By way of specification of Bourguignon's result, we showed that (see \cite{14,20})
 $$
D_1=\frac1{p+1}\ d; \quad 
D_2=\frac1{n-p+1}\ g\wedge d^*; \quad 
D_3=\nabla-\frac1{p+1}\ d - \frac1{n-p+1}\ g\wedge d^* 
$$  
and proved (see \cite{14,20}) that the kernel of the third basis operator $D_3$ consists of conformal Killing $r$-forms. These forms were introduced by Tachibana (see~\cite{21}) and Kashiwada (see\cite{7}) as a natural generalization of \textit{conformal Killing vector fields}, which are also called infinitesimal conformal transformations (see \cite{12}). The space of conformal Killing $r$-forms will be denote by \tm.

In addition, we mention that the kernel of the first basis differential operator  $D_1$ consists of closed $r$-forms and the kernel of the second basis differential operator  $D_2$ consists of co-closed $r$-forms. Two vector spaces of closed and co-closed $r$-forms will be denoted by \dm\  and  \fm.
  
The condition  $\om\in ker D_1\cap ker D_2$ characterizes the form $\om$  as a \textit{harmonic form} (see \cite[p.~204]{10}), therefore the vector space  \hm\ of harmonic \hbox{$r$-forms} is defined as 
$$\hm\  = \dm \cap   \fm .$$

The condition  $\om\in ker D_2\cap ker D_3$ characterizes the $r$-form   as a \textit{co-closed conformal Killing form} which is also called as a \textit{Killing form} (see \cite[p.~65--66]{4},  \cite{18}). Therefore the vector space \km\ of co-closed conformal Killing $r$-forms is defined as 
$$
\km\  = \tm\cap    \fm  . 
$$
Co-closed conformal Killing forms can be considered as a generalization of Killing vector fields that define infinitesimal isometric transformations. 

The condition $\om\in ker D_1\cap ker D_3$  characterizes $\om$  as a \textit{closed conformal Killing form}. Therefore the space \fm\  of closed conformal Killing \hbox{$r$-forms} is defined as (see \cite{20})
$$\fm\  = \dm\cap    \tm  .$$

Closed conformal Killing forms (which are also called planar forms) is a generalization of closed conformal Killing vector fields, which are also called local gradient infinitesimal conformal transformations (see \cite{12}). We denote by \cm\ the subspace of  $\Omega^r(M)$ which consists of parallel $r$-forms, which are also called covariant constant $r$-forms. It is clear that  
$$\cm\ = \km\cap    \pmm \subset  \hm .$$

Using definitions and propositions which we have formulated above we obtain the 3D-diagram of inclusions of subspaces of differential $r$-forms on an $n$-dimensional $(1\leq   r\leq   n - 1)$ Riemannian manifold $(M,g)$.

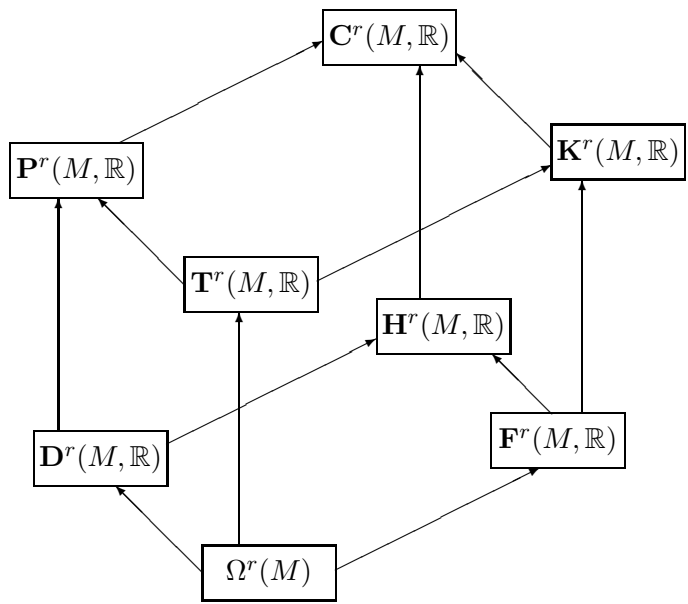
\begin{center}
\unitlength=0.8cm
\begin{picture}(11.5,10)
\put(3.2,0){\framebox(2.2,0.9){\omm}}
\put(0.4,1.9){\framebox(2.2,0.9){$\bold D^r(M,\mathbb{R})$}}
\put(8,2.2){\framebox(2.2,0.9){$\bold F^r(M,\mathbb{R})$}}
\put(6.1,4.1){\framebox(2.2,0.9){$\bold H^r(M,\mathbb{R})$}}

\put(3.2,0.45){\vector(-1,1){1.45}} 
\put(5.4,0.5){\vector(2,1){3.4}}
\put(9,3.1){\vector(-1,1){1}}
\put(2.6,2.6){\vector(2,1){3.48}}

\put(2.9,4.8){\framebox(2.2,0.9){$\bold T^r(M,\mathbb{R})$}}
\put(0,6.7){\framebox(2.2,0.9){$\bold P^r(M,\mathbb{R})$}}
\put(9,7){\framebox(2.2,0.9){$\bold K^r(M,\mathbb{R})$}}
\put(5.2,8.9){\framebox(2.2,0.9){$\bold C^r(M,\mathbb{R})$}}

\put(2.9,5.25){\vector(-1,1){1.45}} 
\put(5.1,5.3){\vector(2,1){3.9}}
\put(9,7.5){\vector(-1,1){1.6}}
\put(1.8,7.6){\vector(2,1){3.4}}

\put(0.8,2.8){\vector(0,1){3.9}}  
\put(3.8,0.9){\vector(0,1){3.9}}
\put(9.5,3.1){\vector(0,1){3.9}}
\put(6.8,5.0){\vector(0,1){3.9}}
\end{picture}

\end{center}
 
 Here, for instance, the arrow \fm\ $\to$    \km\   means that the vector space  \km\   is subspace of \fm.

\section{Harmonic forms and Betti numbers}

One of the motivations to study Tachibana numbers of a compact Riemannian manifold $(M,g)$ is their close relation to Betti numbers. In this section we will present a brief review of the results of the theory of  harmonic forms and Betti numbers. 

The \textit{Laplacian on forms}  $\Delta$:  \omm $\to$ \omm, also called the \textit{Hodge Laplacian}, is defined in terms of the natural Riemannian operators $d$ and $d^*$ as  
$\Delta=d\,d^*+d^*\,d$ (see \cite[p. 52]{1}, \cite[p.~377]{2},  \cite[p.~316]{3},  \cite[p.~204]{10}). The operator  $\Delta$ is nonnegative elliptic second order linear differential operator and its kernel is finite dimensional on a compact Riemannian manifold $(M,g)$. 

Two following equivalent conditions define   $\om\in\omm$ as a \textit{harmonic form} $d\om=0$  and $d^*\om=0$  or  $\Delta\om=0$ on a compact Riemannian manifold $(M, g)$ (see \cite[p.~204]{10}).

If we denote by \hm\  the vector space of harmonic $r$-forms on a compact Riemannian manifold $(M, g)$ then by Hodge theory (see \cite{11}) the $r$-th \textit{Betti number} $b_r(M)$ of $(M, g)$ is defined by the equation $b_r(M)\! =\! \dim \hm\!<\!\infty$.
 
The Hodge Laplacian commutes  $*\,\Delta=\Delta\,*$       with the well known (see \cite[p.~33]{1},  \cite[p.~203]{10}) Hodge star operator  ${*}\!\: \omm \to\Omega^{n-r}(M)$   that implies the following isomorphism   ${*}\!\: \hm \to {\bold H}^{n-r}(M,{\mathbb R})$. From the isomorphism we obtain (see \cite[p.~389]{10}) the \textit{Poincar\'e duality theorem} $b_r(M) = b_{n - r} (M)$. 

\section{Conformal Killing forms and Tachibana numbers}
Before explaining main results of our paper we want to point out that they are extensions of our results in \cite{14,15,16,18,17,19,20} for conformal Killing forms and Tachibana numbers. In this section we will present a brief review of these results.

Throughout in this section we let $(M, g)$ be an $n$-dimensional compact and oriented Riemannian manifold. For any linear differential operator $D$ the inner product structures on  \omm\ allow us to define the formal adjoint operator $D^*$ to $D$ (see \cite[p.~460]{1}). In the case of the first order natural Riemannian (with respect to isometric diffeomorphisms) operator $D$ we can define a \textit{strong Laplacian} $D^*\circ  D$ (see \cite[p.~377]{2}, \cite[p.~316-317]{3}). The first simple property of these operators comes from the fact that any strong Laplacian is a non-negative elliptic second order linear differential operator (see \cite[p.~314-315]{3}) with a finite dimensional kernel (see \cite[pp. 461-463]{1}). For the other properties of these operators, see \cite{2,3,10,25}. 

We showed in \cite{16} that the formal adjoint operator  $D^*_3$ to  $D_3$ determined by the formula  
$D^*_3=\nabla^*-\frac1{p+1}\ d^*-\frac1{n-p+1}\ d\circ trace$
and the Tachibana Laplacian  $\square= D^*_3\circ D_3$   have the form 
\begin{equation}\label{(4.1)}
\square=D^*_3\circ D_3=\frac1{r(r+1)}\
\left(\bar\Delta-\frac1{r+1}\ d^*\circ d-\frac1{n-r+1}\ d\circ d^*
\right).
\end{equation}
The symbol $\bar\Delta$  is called the \textit{Bochner rough Laplacian} and   $\bar\Delta:=\nabla^*\circ\nabla$   where we denote the formal adjoit of  $\nabla$  by  $\nabla^*$ (see \cite[p. 52]{1},  \cite[p.~377]{2}). 

We proved the following three propositions (see \cite{16}):
\begin{align} 
\om\in\tm &\Leftrightarrow \om\in Ker(D^*D); \label{(4.2)}\\
\om\in\km &\Leftrightarrow \om\in Ker(D^*D)
\cap Ker\ d^*; \label{(4.3)}\\
\om\in\pmm &\Leftrightarrow \om\in Ker(D^*D)\cap Ker\ d. \label{(4.4)}
\end{align}
It is known that the kernel of  the Tachibana Laplacian $\square=D^*_3\circ D_3$  has a finite dimension and therefore we concluded that (see \cite{15}) 
 $$
 \begin{array}{l}
 \tm=\dim_{\R}(Ker D^*_3D_3)=t_r(M)<\infty;\\
 \km=k_r(M)<\infty;\\
 \pmm=p_r(M)<\infty
 \end{array}
 $$
on a compact Riemannian manifold $(M, g)$. The numbers $t_r(M)$, $k_r(M)$ and $p_r(M)$  we have called the \textit{Tachibana number}, the \textit{Killing number} and the \textit{planarity number} of a compact Riemannian manifold $(M, g)$, respectively (see~\cite{17}). These numbers satisfy the following duality properties $t_r (M)  =  t_{n - r} (M)$  and  $p_r (M)  =  k_{n - r} (M)$ for all $r = 1, \dots, n - 1$. These equalities are analogues of the Poincare duality for Betti numbers and corollaries of the following isomorphisms (see \cite{8,14,20})  
$$*\,{:}\  \tm\to {\bold T}^{n-r}(M,\R)        \hbox{ and }    {*}\,{:}\  \pmm\to {\bold K}^{n-r}(M,\R).$$        
Moreover, Tachibana numbers $t_r(M)$ are conformal scalar invariants, the Killing number $k_r(M)$ and the planarity number $p_r(M)$ are projective scalar invariants of a Riemannian manifold $(M, g)$ for all $r = 1,\dots, n - 1$. In addition, we mention that the first proposition is a corollary of conformal invariance of conformal Killing $r$-forms (see \cite{5}). The second proposition of our theorem is a corollary of projective invariance of closed and co-closed conformal Killing $r$-forms (see \cite{18}).

Let $(M, g)$ be an $n$-dimensional connected Riemannian manifold then the Tachibana number $t_r(M)$, the Killing number $k_r(M)$ and the planarity number $p_r(M)$ of $(M, g)$ satisfy the following inequalities 
$$
0\leq t_r(M)=\frac{(n+2)!}{(r+1)!\,(n-r+1)!};\qquad
0\leq k_r(M)=\frac{(n+1)!}{(r+1)!\,(n-r)!}
$$
$$
0\leq p_r(M)=\frac{(n+1)!}{r!\,(n-r+1)!};
$$ 
for all $r = 1,\dots, n -  1$. Moreover, any of two numbers $p_r(M)$ and $k_r(M)$ is maximal if and only if $(M, g)$ is Riemannian manifold with positive constant curvature (see \cite{9,18,20}). In addition, the Tachibana number $t_r(M)$ is maximal if and only if $(M, g)$ is conformal flat Riemannian manifold (see \cite{7,13}). 

\section{Tachibana and Betti numbers}
In this section we will show some relationships between Tachibana and Betti numbers. 

 First, we formulate \textit{vanishing theorems} of Betti and Tachibana numbers. Let $(M, g)$ be an $n$-dimensional compact oriented Riemannian manifold and  $\RR\,{:}\  \Omega^2(M)\to\Omega^2(M)$ be the standard symmetric \textit{Riemannian curvature operator} of $(M, g)$ (see \cite[p.~35--36]{10}). 
 
 Supposing  $\RR\geq0$, we have $b_r(M)\leq\frac{n!}{r!\,(n-r)!}=b_r(T^n)$  where $T_n$ is a flat Riemannian $n$-torus and $r = 1,\dots, n-1$ (see \cite[p.~212]{10}). Moreover, an $n$-dimensional compact Riemannian manifold with the positive curvature operator $\RR$  is a spherical space form (see \cite{6}) and its Betti numbers $b_1(M )$, $\dots$, $b_{n - 1}(M)$ are zeros (see \cite[p.~212]{10}). In addition, we proved (see \cite{15}) that in this case an arbitrary conformal Killing $r$-form   $\om$ is uniquely decomposed in the form  $\om'+\om''$   where  $\om'$   is a Killing and  $\om''$ is a closed conformal Killing $r$-forms on $(M, g)$ for all 
 $r = 1, \dots , n -  1$ and hence $t_r (M) = k_r (M) + p_r (M)$.
 
On the other hand, if $\RR$  is non-positive then $t_r(M)\leq\frac{n!}{r!\,(n-r)!}=t_r(T^n)$  and if $\RR<0$  somewhere, then  $t_r(M)=0$ for $r = 1,\dots, n - 1$ (see \cite{15}). 

Second, we consider Betti and Tachibana numbers of a compact conformally flat Riemannian manifold and prove the following theorem as a corollary of above results.
%
\begin{theorem}
Let $(M, g)$ be an $n$-dimensional $(n\geq3)$ compact conformally flat Riemannian manifold with the positive or negative definite Ricci tensor Ric then 
$t_r(M)\cdot   b_h(M) = 0$ for all  $h,r=1,\dots,n-1$.
\end{theorem}
\begin{proof} 
Let one the two following conditions be satisfied $Ric > 0$ or $Ric < 0$ at every point of $(M, g)$. If we suppose that $(M, g)$ is conformally flat Riemannian manifold of dimension $n\geq3$, then in the first case we have  $\RR > 0$ and Betti numbers  $b_1(M)=\cdots=b_{n-1}(M)=0$ because there are no non-zero harmonic $r$-forms for all  $r=1,\dots,n-1$ (see \cite[pp.~78-79]{4}). On the other hand, in the second case we have  $\RR < 0$ and Tachibana numbers  $t_1(M)=\cdots=t_{n-1}(M)=0$  because there are no non-zero conformal Killing $r$-forms for all $r=1,\dots,n-1$  (see \cite{19}). Therefore we conclude that 
$t_r(M)\cdot b_h(M) = 0$ for all  $h,r=1,\dots,n-1$, which finishes the proof of the theorem.
\end{proof}
Next, we consider the case of an even-dimensional compact and oriented conformally flat Riemannian manifold. In this special case the following theorem is true.
%
\begin{theorem} 
Let $(M, g)$ be a 2$r$-dimensional compact and oriented conformally flat Riemannian manifold. Then the following propositions are true. 
\begin{enumerate}
\item	If the scalar curvature $s\neq0$ and the Tachibana number $t_r (M)\neq 0$ then the Betti number $b_r (M) = 0$. Moreover, if s is a positive constant, then $t_r (M) = k_r (M) + p_r (M)$ for 
$k_r(M)=\dim_{\R}(Ker D^*_3D_3\cap Im\, d^*)$\\
  and 
$p_r(M)=\dim_{\R}(Ker D^*_3D_3\cap Im\, d)$.
\item	If the scalar curvature $s\neq   0$ and the Betti number $b_r (M)\neq   0$ then the Tachibana number $t_r (M) = 0$.
\item	If the scalar curvature $s = 0$ then the Betti number $b_r (M)\neq   0$ if and only if the Tachibana number $t_r (M) \neq 0$ and moreover  \\
$t_r (M)=b_r (M)=(2r)!\,(r!)^{-2}$.
\end{enumerate} 
\end{theorem}
\begin{proof} 
First of all, we recall, that the Hodge Laplacian  operator $\Delta$  admits the well known \textit{Weitzenb\"ock decomposition} (see \cite[p.~53]{1})
\begin{equation}\label{(5.1)}
\Delta\om=\bar\Delta\om+F_r(\om),
\end{equation}
where (see \cite[pp.~60--61]{4})
$$
F_r(\om)(X_1,\dots,X_r)=\sum_{\a=1}^r Ric(e_j,X_\a)\om(X_1,\dots,X_{\a-1},e_j,
X_{\a+1},\dots,X_r) -
$$ 
$$
\sum_{\a<\b}^{1,\dots, r} R(e_j,e_k,X_\a,X_\b)\om(X_1,\dots,X_{\a-1},e_j,
X_{\a+1},\dots,X_{\b-1},e_k,
X_{\b+1},\dots,X_r) 
$$
for an arbitrary \ $\om\in\Omega^r(M)$; \ $X_1,\dots,X_r\in C^\infty TM$ \  any orthonormal basis  $\{e_1,\dots,e_n\}$, and the curvature tensor $R$ and the Ricci $Ric$ of 
$(M, g)$. As a consequence of this fact, we obtain from \eqref{(4.1)} the following 
\begin{equation}\label{(5.2)}
\square\,\om=\frac{1}{r(r+1)}\ \left(\Delta\om- F_r(\om)-\frac r{r+1}\ d^*d\om-
\frac{n-r}{n-r+1}d\,d^*\om\right)
\end{equation}
for any  $\om\in\omm$. In particular, we have 
\begin{equation}\label{(5.3)}
\Delta\om=F_r(\om)+\frac r{r+1}\ d^*d\om +
\frac{n-r}{n-r+1}d\,d^*\om
\end{equation}
for an arbitrary conformal Killing $r$-form  $\om$.

Next, if we suppose that our Riemannian manifold $(M, g)$ has the dimension $2r$ and can be
reduced to an Euclidean space by suitable conformal transformation the metric tensor $g$, then we have the equality (see \cite[p.~79]{4})
\begin{equation}\label{(5.5)}
F_r(\om)=\frac 1{2(2r-1)}\ s\cdot\om ,
\end{equation}
After that, from \eqref{(5.3)} and \eqref{(5.5)} we can deduce the equation 
\begin{equation}\label{(5.6)}
\Delta\om=\frac{r+ 1}{2r(2r-1)}\ s\cdot\om ,
\end{equation}
where $s$ is the scalar curvature of $(M, g)$. In particular, if $s\neq0$  we deduce from \eqref{(5.6)} that $\Delta\om=0$  if and only if  $\om=0$. Moreover, if we suppose that $s$ is a positive constant then we can rewrite the equation \eqref{(5.6)} in the form
\begin{equation}\label{(5.7)}
\om= d^*d\om' +d\,d^*\om'
\end{equation}
where  $\om'=\frac{2r(2r-1)}{(r+ 1)s}\ \om$. The equality \eqref{(5.7)} is the well known Hodge-de Ram decomposition where there is no a harmonic form. From \eqref{(5.7)} we conclude the following orthogonal decomposition
\begin{equation}\label{(5.8)}
\tm=\pmm \oplus\km
\end{equation}
where
\begin{equation}\label{(5.9)}
\km=\{\om\in\omm\,|\,\om\in Ker D^*_3D_3\cap Im\,d^*\}
\end{equation}
and 
\begin{equation}\label{(5.10)}
\pmm=\{\om\in\omm\,|\,\om\in Ker D^*_3D_3\cap Im\,d\}
\end{equation} 
From \eqref{(5.8)} -- \eqref{(5.10)} we obtain the following equalities $t_r(M)= k_r(M)+ p_r(M)$, 
$k_r(M)=\dim_{\R}(Ker D^*_3D_3\cap Im\,d^*)$  and  $p_r(M)=\dim_{\R}(Ker D^*_3D_3\cap Im\,d)$ which finish the proof of the first assertion of the theorem. 

Next, from \eqref{(4.1)}, \eqref{(5.1)} and \eqref{(5.5)} we infer 
\begin{equation}\label{(5.11)}
\frac s{2(2r-1)}\ \om=\frac r{r+1}\ \Delta \om- r(r+1)\ \square\,\om
\end{equation} 
for any  $\om\in\omm$. 

Then for an arbitrary harmonic $r$-form  $\om$ we have 
\begin{equation}\label{(5.12)}
\frac s{2(2r-1)}\ \om=- r(r+1)\ \square\,\om.
\end{equation} 
In particular, if $s\neq0$  we deduce from \eqref{(5.6)} that $\square\,\om  = 0$ if and only if  $\om  = 0$. 

Finally, we consider a 2$r$-dimensional compact conformally flat Riemannian manifold with zero scalar curvature. In this case, from \eqref{(5.11)} we obtain   $\Delta\om=(r+1)^2\,\square\,\om$. If we suppose that $b_r (M)\neq   0$ then there exists a nonzero harmonic $r$-form $\om$  such that  $\Delta\om=(r+1)^2\,\square\,\om  =    0$. Then  $\om$ is conformal Killing and from \eqref{(5.1)} we obtain the equation  $\nabla\om=0$.  In this case we have  
$b_r (M)=t_r(M)=\frac{n!}{r!\,(n-r)!}=(2r)!\cdot(r!)^{-2}$, because a parallel form is completely determined by e value at point. This finishes the proof of the theorem.
\end{proof}

As a consequence of the previous theorem we obtain the following corollary.
\begin{corollary}
 Let $(M, g)$ be a 2$r$-dimensional compact and oriented conformally flat Riemannian manifold with constant nonzero scalar curvature s. Then any nonzero conformal Killing $r$-form is an eigenform of the Hodge Laplacian  $\Delta$ corresponding to the eigen-value $\la=(r+1)\,(2r(2r-1))^{-1}s$ for $s > 0$ and any nonzero harmonic $r$-form is an eigenform of the Tachibana Laplacian $\square$ corresponding to the eigenvalue  $\la=-((2r(r+1)(2r-1))^{-1}s$ for $s < 0$.
\end{corollary}
\begin{proof} 
It is well known, if a nonzero $r$-form $\om$  satisfies the equation  $\Delta\om=\la\,\om$ for a constant  $\la > 0$, it is called an eigenform of the Laplacian  $\Delta$ corresponding to the eigen-value $\la$. Then if we suppose that $s$ is a positive constant, then from \eqref{(5.6)} we conclude that the conformal Killing $r$-form  $\om$ is an eigenform of  $\Delta$ corresponding to the eigenvalue  $\la=(r+1)\,(2r(2r-1))^{-1}s$ (see also \cite{22}).  Moreover, if we suppose that s is a negative constant, then from \eqref{(5.12)} we conclude that the harmonic $r$-form  $\om$  is an eigenform of $\square$ corresponding to the eigenvalue $\la=-((2r(r+1)(2r-1))^{-1}s$. 
\end{proof}

Next, we consider conformal Killing and harmonic forms on a Riemannian manifold of constant sectional curvature which is an example of a conformally flat Riemannian manifold.

\begin{theorem}
Let $(M, g)$ be an $n$-dimensional compact and oriented Riemannian manifold of constant nonzero sectional curvature C. Then we have three propositions. 
\begin{enumerate}
\item A nonzero $r$-form  $\om$ is harmonic if and only if  $\om$ is an eigenform of the Tachibana Laplacian $\square$ corresponding to the eigenvalue  $\la=-(n-r)(r+1)^{-1}C$ for $C < 0$.
\item Any nonzero closed conformal Killing $r$-form $\om$  is an eigenform of the Hodge Laplacian~$\Delta$ corresponding to the eigenvalue $\la=r(n-r+1)\,C$ for $C > 0$.  For $n < 2r$ converse is true.
\item Any nonzero Killing $r$-form  $\om$ is an eigenform of the Hodge Laplacian~$\Delta$  corresponding to the eigenvalue $\la=(n-r)(r+1)\,C$ for $C > 0$.   For $n > 2r$ converse is true.
\end{enumerate}
\end{theorem}
\begin{proof}
 Let $(M, g)$ be an $n$-dimensional compact and oriented Riemannian manifold with constant sectional curvature $C$ then (see \cite{8}) we have the identity
\begin{equation}\label{(5.13)}
F_r(\om)=r(n-r)\,C\cdot\om
\end{equation}
for an arbitrary  $\om\in\omm$. In this case we can rewrite \eqref{(5.2)} in the form
\begin{equation}\label{(5.14)}
\llap{$\square$}\,\om=\frac1{r(r+1)} \left(\Delta\om-
r(n-r)\,C\cdot\om-
\frac r{r{+}1}\ d^*d\,\om - \frac {n-r}{n{-}r{+}1}\ d\,d^*\,\om
\right).
\end{equation}
From this, we see $d\om=0$  and  $d^*\om=0$ hold if and only if $$\square\,\om=(n-r)(r+1)^{-1}C\cdot\om$$  holds. On the other hand, we can rewrite \eqref{(5.2)} in two different forms
\begin{equation}\label{(5.15)}
\llap{$\square$}\,\om=\frac1{r(r{+}1)} \left(\frac{1}{n{-}r{+}1}\,\Delta\om
-
 r(n{-}r)\,C\cdot\om +\frac {n-2r}{(r\,{+}\,1)(n\,{-}\,r\,{+}\,1)}\, d^*d\om
\right)\hspace{-2mm}
\end{equation}
and
\begin{equation}\label{(5.16)}
\llap{$\square$}\,\om=\frac1{r(r{+}1)} \left(\frac{1}{r{+}1}\,\Delta\om -
 r(n-r)\,C\cdot\om - \frac {n-2r}{(r{+}1)(n{-}r{+}1)}\, dd^*\om
\right)\!.
\end{equation}
From \eqref{(5.15)} we can conclude that if $d\om=0$  and $\square\,\om=0$  hold then   $\Delta\om=(n-r)(r+1)\,C\cdot\om$     holds. At the same time, from \eqref{(5.16)} we conclude that if $d^*\om=0$  and $\square\,\om=0$  hold then  $\Delta\om=(n-r+1)r\,C\cdot\om$       holds. It is easy to see that converses are true only for cases when $n < 2r$ and $n>2r$, respectively. 
\end{proof}

In the following theorem we study compact Riemannian manifolds whose Ricci tensor is negative and positive semi-definite and we give the lower bound for the Betti number $b_r(M)$ and the Tachibana number $t_r(M)$, respectively. More precisely, we have
%
\begin{theorem}
Let $(M, g)$ be a compact $n$-dimensional Riemannian manifold satisfy one of the following conditions:
\begin{enumerate}
\item	the Ricci tensor Ric of  $(M, g)$  is negative semi-define and the first Tachibana number $t_1(M)=h\leq n$;
\item	the Ricci tensor Ric of  $(M, g)$  is positive semi-define and the Betti number $b_1(M)=h\leq   n$. Then
 $$
 \frac{h!}{r!\,(h-r)}\leq b_r(M)=t_r(M)\leq \frac{n!}{r!\,(h-r)}
 $$       
 for $1\leq   r < h = b_1(M) \leq   n$. If $Ric \leq  0$ and $t_1(M)= n$ or $Ric\geq0$ and $b_1(M)=n$ then $(M, g)$ is a flat Riemannian $n$-torus. 
 \end{enumerate}
\end{theorem}
\begin{proof} Let  $(M, g)$  be a compact $n$-dimensional Riemannian manifold and\linebreak 
$t_1(M) = h \neq  0$ then there are $h$ linearly independent nonzero conformal Killing 1-forms   $\om_1,\dots,\om_h$. We denote by  $X_1,\dots,X_h$ the dual conformal Killing vector fields, i.e. $\om_a(Y)=g(Y,X_a)$  for  $a=1,\dots,h$. If we suppose that the Ricci tensor $Ric$ is negative semi-definite then $X_1,\dots,X_h$  are parallel (see \cite[p.~53--55]{4}), i.e.  $\nabla X_1=0,\dots,\nabla X_h=0$. In this case, following 1-forms $\om_1,\dots,\om_h$  are parallel and hence are conformal Killing 1-forms. Using this result, we conclude that  $\theta_{i_1i_2\cdots i_r}= \om_{i_1}\wedge\om_{i_2}\wedge\cdots,\wedge\om_{i_r}$   for any $1\leq   i_1 <  \cdots < i_r \leq   h$ are parallel $r$-forms and
hence conformal Killing $r$-forms. 
Since these  $\theta_{i_1i_2\cdots i_r}$ are  $\ds\frac{h!}{r!\,(h-r)!}$ linearly independent conformal Killing $r$-forms then we have 
$\ds t_r(M)\geq\frac{h!}{r!\,(h-r)!}$
   for $r < h$. Moreover, we have $\ds b_r(M)\geq\frac{h!}{r!\,(h-r)!}$   for  $r < h$ because an arbitrary parallel $r$-form  $\om$ is harmonic. On the other hand, let $(M, g)$ be a compact $n$-dimensional Riemannian manifold with $b_1(M) = h\neq 0$ then there are $h$ linearly independent nonzero harmonic 1-forms   $\om_1,\dots,\om_h$. We denote by $X_1,\dots,X_h$  the dual harmonic vector fields, i.e. $\om_a(Y)=g(Y,X_a)$  for  $a=1,\dots,h$.   If we suppose that the Ricci tensor $Ric$ is positive semi-definite then  $X_1,\dots,X_h$ are parallel (see \cite[p.~53--55]{4}), i.e.  $\nabla X_1=0,\dots,\nabla X_h=0$. From this we can conclude that  
   $\ds b_r(M)\geq\frac{h!}{r!\,(h-r)!}$  for  $r < h$ and hence   
   $\ds t_r(M)\geq\frac{h!}{r!\,(h-r)!}$. Next, if $Ric\leq   0$ and $t_1(M)=n$ or $Ric\geq0$ and $b_1(M)=n$, then $(M, g)$ has a parallel frame. The curvature tensor of $(M, g)$ vanishes in this frame, so $(M, g)$ is flat torus (see \cite[p. 208]{10}) and hence  $\ds b_r(M)=t_r(M)=\frac{h!}{\smash{r!\,(h-r)!}}$. This concludes the proof of the theorem.
\end{proof}

Finally, we prove the following theorem and formulate its corollary.
%
\begin{theorem}
Let $(M, g)$ be a compact $n$-dimensional Riemannian manifold with its zero first Betti number $b_1(M)$  and non-zero first planarity number  
$p_1(M)$  then  Betti  numbers 
$b_2(M),\dots, b_{n-1}(M)$ are equal to zero too.
\end{theorem}
\begin{proof}  
If $b_1(M) = 0$ and   $p_1(M) \neq 0$  then there exists a non-zero exact conformal Killing 1-form $\om=d\,f$  such that $\nabla{\rm grad} f=-n^{-1}\Delta f\cdot g$  for a smooth function~$f$.  In this case, due to  Tashiro theorem (see \cite{23}), $(M, g)$ is conformally diffeomorphic to a Euclidean sphere ${\mathbb S}^n$ and hence $b_r(M)=b_r({\mathbb S}^n)=0$  for $r = 1,\dots, n - 1$. This finishes the proof.
\end{proof}
%
\begin{corollary}
Let $(M, g)$ be a compact $n$-dimensional Riemannian manifold such that its Betti number $b_{n -1}(M) = 0$  and Killing number  $k_{n -1}(M) \neq 0$  then  Betti  numbers $b_1(M),\dots,b_{n-2}(M)$ are equal to zero too.
\end{corollary}

\noi\textbf{Acknowledgment.} The paper was supported by grant P201/11/0356 of the Czech Science Foundation.

\end{document}